\documentclass[10pt]{amsart}
\usepackage{amssymb,amsbsy,amsmath,amsfonts,times, amscd}
\usepackage{latexsym,euscript,exscale}
\usepackage{graphicx,color} \usepackage{amsthm}
\usepackage{fancyhdr} \usepackage{url}
\usepackage{epigraph}\usepackage{lmodern}
\usepackage{mathrsfs}\usepackage{cancel}
\usepackage[toc,page]{appendix}\usepackage[T1]{fontenc}
\usepackage[all,cmtip]{xy}\usepackage{mathtools}
\usepackage{enumerate}\usepackage{setspace}
\usepackage{helvet}%&LaTeX 
\usepackage{float}
\usepackage[all]{xy}
		
\numberwithin{equation}{section}
\newtheorem{theorem}{Theorem}
\newtheorem{thm}[theorem]{Theorem}
\newtheorem{cor}[theorem]{Corollary}
\newtheorem{lemma}[theorem]{Lemma}
\newtheorem{prop}[theorem]{Proposition}
\theoremstyle{definition}
\newtheorem{defi}[theorem]{Definition}
\newtheorem{problem}[theorem]{Problem}
\theoremstyle{remark}
\newtheorem{remark}[theorem]{Remark}

\numberwithin{theorem}{section}

\newcommand {\M}{\mathbb M}

\newcommand {\C}{\mathbb C}

\newcommand{\eps}{\varepsilon}
\newcommand{\E}{\mathbb{E}}

\newcommand{\cQ}{\mathcal{Q}}

\newcommand{\Z}{\mathbb{Z}}
\newcommand{\Q}{\mathbb{Q}}
\newcommand{\N}{\mathbb{N}}
\newcommand{\R}{\mathbb{R}}

\begin{document}

\title[On weaker notions of nonlinear embeddings.]{On weaker notions of nonlinear embeddings between Banach spaces.}
\author{B. M. Braga }
\address{Department of Mathematics, Statistics, and Computer Science (M/C 249)\\
University of Illinois at Chicago\\
851 S. Morgan St.\\
Chicago, IL 60607-7045\\
USA}\email{demendoncabraga@gmail.com}
\date{}
\maketitle

\begin{abstract}
In this paper, we study nonlinear embeddings between Banach spaces. More specifically, the goal of this paper is to study weaker versions of coarse and uniform embeddability, and to provide suggestive evidences that those weaker embeddings may be stronger than one would think. We do such by proving that many known results regarding coarse and uniform embeddability remain valid for those weaker notions of embeddability.
\end{abstract}

\section{Introduction.}\label{sectionintro}

The study of Banach spaces as metric spaces has recently increased significantly, and much has been done regarding the uniform and coarse theory of Banach spaces in the past two decades. In particular, the study of coarse and uniform embeddings has been receiving a considerable amount of attention (e.g., \cite{B}, \cite{K}, \cite{Ka}, \cite{MN}, \cite{No}). These notes are dedicated to the study of several different notions of nonlinear embeddings between Banach spaces, and our main goal  is to provide the reader with evidences that those kinds of embeddings may not be as different as one would think. 

 Let $(M,d)$ and $(N,\partial)$ be metric spaces, and consider a map $f:(M,d)\to (N,\partial)$. For each $t\geq 0$, we define the \emph{expansion modulus of $f$} as 

$$\omega_f(t)=\sup\{\partial(f(x),f(y))\mid d(x,y)\leq t\},$$\hfill

\noindent  and the \emph{compression modulus of $f$} as

$$\rho_f(t)=\inf\{\partial(f(x),f(y))\mid  d(x,y)\geq t\}.$$\hfill

\noindent Hence, $\rho_f(d(x,y))\leq \partial(f(x),f(y))\leq \omega_f(d(x,y))$, for all $x,y\in M$. The map  $f$ is  uniformly continuous if and only if  $\lim_{t\to 0_+}\omega_f(t)=0$, and $f^{-1}$ exists and it is uniformly continuous if and only if $\rho_f(t)>0$, for all $t>0$. If both $f$ and its inverse $f^{-1}$ are uniformly continuous, $f$ is called a \emph{uniform embedding}. The map $f$ is called \emph{coarse} if $\omega_f(t)<\infty$, for all $t\geq 0$, and \emph{expanding} if $\lim_{t\to\infty}\rho_f(t)=\infty$. If $f$ is both expanding and coarse, $f$ is called a \emph{coarse embedding}. A map which is both a uniform and a coarse embedding is called a \emph{strong embedding}.

Those notions of embeddings  are fundamentally very different, as  coarse embeddings  deal with the large scale geometry of the metric spaces concerned, and uniform embeddings only deal with their local (uniform) structure. Although those notions are fundamentally different, it is still not known whether the existence of those embeddings are equivalent in the Banach space setting. Precisely, the following problem remains open.

\begin{problem}\label{mainproblem}
Let $X$ and $Y$ be Banach spaces. Are the following equivalent?

\begin{enumerate}[(i)]
\item $X$ coarsely embeds into $Y$.
\item $X$ uniformly embeds into $Y$.
\item $X$ strongly embeds into $Y$.
\end{enumerate}
\end{problem}

It is known that Problem \ref{mainproblem} has a positive answer if $Y$ is either $\ell_\infty$ (see \cite{Ka4}, Theorem 5.3) or $\ell_p$, for $p\in [1,2]$ (see  \cite{No}, Theorem 5, and \cite{Ra}, page 1315). C. Rosendal made some improvements on this problem  by showing that if $X$ uniformly embeds into $\ell_p$, then $X$ strongly embeds into $\ell_p$, for all $p\in [1,\infty)$. This result can be generalized by replacing $\ell_p$ with any minimal Banach space (see \cite{B}, Theorem 1.2(i)).

Natural weakenings for the concepts of coarse and uniform embeddings were introduced in \cite{Ro}, and, as it turns out, those weaker notions are rich enough for many applications. Given a map $f:(M,d)\to(N,\partial)$ between metric spaces, we say that $f$ is \emph{uncollapsed} if there exists some $t>0$ such that $\rho_f(t)>0$. The map $f$ is called \emph{solvent} if, for each $n\in\N$, there exists $R>0$, such that 

$$d(x,y)\in [R,R+n] \ \ \text{ implies }\ \  \partial(f(x),f(y))>n,$$\hfill

\noindent for all $x,y\in M$. For each $t\geq 0$, we define the \emph{exact compression modulus of $f$} as

$$\overline{\rho}_f(t)=\inf\{\partial(f(x),f(y))\mid d(x,y)=t\}.$$\hfill

\noindent The map $f$ is called \emph{ almost uncollapsed } if there exists some $t>0$ such that $\overline{\rho}_f(t)>0$.  

It is clear from its definition, that  expanding maps are both solvent and uncollapsed. Also, as $\rho_f(t)\leq \overline{\rho}_f(t)$, for all $t\in[0,\infty)$, uncollapsed maps are also  almost uncollapsed. Therefore, we have Diagram \ref{DiagEx}.

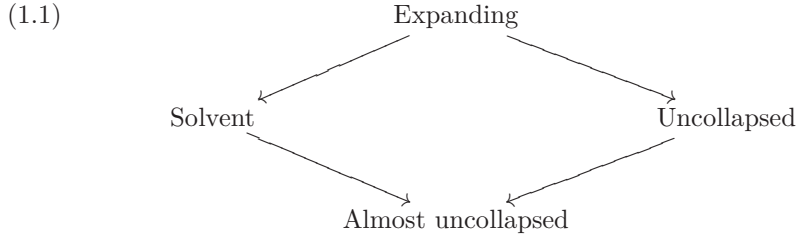
\begin{equation}\label{DiagEx}
    \xymatrix{ 
         & \text{Expanding}  \ar[dl]\ar[dr] &  \\
        \text{Solvent} \ar[dr]& & \text{Uncollapsed}\ar[dl]\\
         & \text{ Almost uncollapsed } &   }
\end{equation}\hfill
	
As a map $f:(M,d)\to (N,\partial)$ has uniformly continuous inverse if and only if $\rho_f(t)>0$, for all $t>0$,  Diagram \ref{DiagUCI} holds.	
	
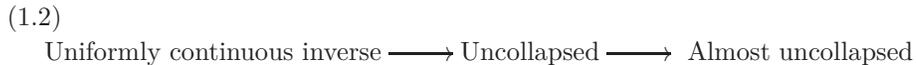
\begin{equation}\label{DiagUCI}
    \xymatrix{ \text{Uniformly continuous inverse} \ar[r]&
        \text{Uncollapsed}\ar[r] &\text{ Almost uncollapsed }  }\end{equation}\hfill
	
None of the arrows in neither Diagram \ref{DiagEx} nor Diagram \ref{DiagUCI} reverse. Indeed, any bounded uniform embedding is uncollapsed (resp. almost uncollapsed), but it is not expanding (resp. solvent). Examples of uncollapsed maps which are not uniformly continuous are easy to be constructed, as you only need to make sure the map is not injective. At last, Proposition \ref{solventcollapsedmaps} below provides an example of a map which is solvent but collapsed (i.e., not uncollapsed), which covers the remaining arrows.

In \cite{Ro}, Theorem 2, Rosendal showed that if there exists a uniformly continuous uncollapsed map $X\to Y$ between Banach spaces $X$ and $Y$, then $X$ strongly embeds into $\ell_p(Y)$, for any $p\in [1,\infty)$. Rosendal also showed that there exists no map $c_0\to E$ which is both coarse and solvent (resp. uniformly continuous and almost uncollapsed), where $E$ is any reflexive Banach space (see \cite{Ro}, Proposition 63 and Theorem 64). This result is a strengthening of a result of Kalton that says that $c_0$ does not  coarsely embed (resp. uniformly embed) into any reflexive space (see \cite{K}, Theorem 3.6).

Those results naturally raise the following question. 

\begin{problem}\label{mainproblemPartII}
Let $X$ and $Y$ be Banach space. Are the statements in Problem \ref{mainproblem} equivalent to the following weaker statements?

\begin{enumerate}[(i)]\setcounter{enumi}{3}
\item $X$ maps into $Y$ by a map which is coarse and solvent.
\item $X$ maps into $Y$ by a map which is uniformly continuous and almost uncollapsed. 
\end{enumerate}
\end{problem}

In these notes, we will not directly deal with Problem \ref{mainproblem} and Problem \ref{mainproblemPartII} for an arbitrary $Y$, but instead, we intend to provide the reader with many suggestive evidences that those problems either have a positive answer or that any possible differences between the aforementioned embeddings are often negligible. 

We now describe the organization and the main results of this paper. In Section \ref{sectionbackground}, we give all the remaining notation and background necessary for these notes. Also, in Subsection \ref{subsectionEx} we give an example of a  map $\R\to \ell_2(\C)$  which is Lipschitz, solvent and collapsed (Proposition \ref{solventcollapsedmaps}). 

For a Banach space  $X$, let $q_X=\inf\{q\in [2,\infty)\mid X\text{ has cotype }q\}$ (see Subsection \ref{deftype} for definitions regarding type and cotype).  In \cite{MN}, M. Mendel and A. Naor proved that if a Banach space $X$ either coarsely or uniformly embeds into a Banach space $Y$ with nontrivial type, then $q_X\leq q_Y$ (see Theorem 1.9 and Theorem 1.11 of \cite{MN}). In Section \ref{sectioncotype}, we prove the following strengthening of this result.

\begin{thm}\label{solventemb}
Let $X$ and $Y$ be Banach spaces, and assume that $Y$ has nontrivial type. If either 

\begin{enumerate}[(i)]
\item there exists a coarse solvent map $X\to Y$, or
\item there exists a uniformly continuous  almost uncollapsed   map $X\to Y$,
\end{enumerate}

\noindent then, $q_X\leq q_Y$.
\end{thm}

Theorem \ref{solventemb} gives us the following corollary.

\begin{cor}\label{funnyOS}
Let $p,q\in [1,\infty)$ be such that $q>\max\{2,p\}$. Any uniformly continuous map  $f:\ell_q\to \ell_p$ (resp. $f:L_q\to L_p$) must satisfy

$$\sup_t\inf_{\|x-y\|=t}\| f(x)-f(y)\|=0.$$\hfill
\end{cor}

While the unit balls of the $\ell_p$'s are all uniformly homeomorphic to each other (see \cite{OS}, Theorem 2.1), Corollary \ref{funnyOS} says that those uniform homeomorphisms cannot be extended in any reasonable way.

In Section \ref{subsectionPropQ}, we look at Kalton's Property $\cQ$, which was introduced in \cite{K}, Section 4. We prove that Property $\cQ$ is stable under those weaker kinds of embeddings (see Theorem \ref{PropertyQ}). Although the stability of Property $\cQ$ under coarse and uniform embeddings is implicit in \cite{K}, to the best of our knowledge, this is not explicitly written in the literature.

Theorem \ref{PropertyQ} allows us to obtain the following result (see Theorem \ref{IntoSuperreflexive2} below for a stronger result).

\begin{theorem}\label{IntoSuperreflexive}
Let $X$ and $Y$ be Banach spaces, and assume that $Y$ is reflexive (resp. super-reflexive). If either

\begin{enumerate}[(i)]
\item there exists a coarse solvent map $X\to Y$, or
\item there exists a uniformly continuous  almost uncollapsed   map $X\to Y$,
\end{enumerate}

\noindent then, $X$ is either reflexive (resp. super-reflexive) or $X$ has a spreading model equivalent to the $\ell_1$-basis (resp. trivial type). 
\end{theorem}

Theorem \ref{IntoSuperreflexive} was proven in \cite{K}, Theorem 5.1, for uniform and coarse embeddings into  super-reflexive spaces. Although the  result above for uniform and coarse embeddings into reflexive spaces is implicit in \cite{K}, we could not find this result explicitly written anywhere in the literature. 

It is worth noticing that Theorem \ref{IntoSuperreflexive} cannot be improved for embeddings of $X$ into super-reflexive spaces in order to guarantee that $X$   either is super-reflexive or has a spreading model equivalent to the $\ell_1$-basis (see Remark \ref{Gideon}). 

As mentioned above, Problem \ref{mainproblem} has a positive answer for $Y=\ell_p$, for all $p\in [1,2]$ (see \cite{No}, Theorem 5, and \cite{Ra}, page 1315). In Section \ref{SectionHilbert}, we show that Problem \ref{mainproblemPartII} also has a positive answer in the same settings. Precisely, we show the following. 

\begin{thm}\label{ThmHilbert}
Let $X$ be a Banach space, and $Y=\ell_p$, for any $p\in [1,2]$. Then Problem \ref{mainproblemPartII} has a positive answer.
\end{thm}

See Theorem \ref{ThmHilbertCru} below for another equivalent condition to the items in Problem \ref{mainproblemPartII} in the case  $Y=\ell_p$, for  $p\in [1,2]$.

At last, in Section \ref{Sectionlinfty}, we give a positive answer to Problem \ref{mainproblemPartII} for $Y=\ell_\infty$. This is a strengthening of Theorem 5.3 of \cite{Ka4}, where Kalton shows that Problem \ref{mainproblem} has a positive answer for $Y=\ell_\infty$. Moreover, Kalton showed that uniform (resp. coarse) embeddability into $\ell_\infty$ is equivalent to Lipschitz embeddability.

\begin{thm}\label{Thmlinfty}
Let $X$ be a Banach space, and $Y=\ell_\infty$. Then Problem \ref{mainproblemPartII} has a positive answer. Moreover, for $Y=\ell_\infty$, items (iv) and (v) of Problem \ref{mainproblemPartII} are also equivalent to Lipschitz   embeddability into $\ell_\infty$.
\end{thm}

Even though we do not give a positive answer to Problem \ref{mainproblem} and Problem \ref{mainproblemPartII}, we believe that the aforementioned results provide considerable suggestive evidences that all the five different kinds of embeddings $X\hookrightarrow Y$ above preserve the geometric properties of $X$ in a similar manner. 

\section{Preliminaries.}\label{sectionbackground}

In these notes, all the Banach spaces are assumed to be over the reals, unless otherwise stated. If $X$ is a Banach space, we denote by $B_X$ its closed unit ball. Let $(X_n,\|\cdot\|_n)_{n=1}^\infty$ be a sequence of Banach spaces. Let $\mathcal{E}=(e_n)_{n=1}^\infty$ be a $1$-unconditional basic sequence in a Banach space $E$ with norm $\|\cdot\|_E$. We define the sum $(\oplus _n X_n)_{\mathcal{E}}$ to be the space of sequences $(x_n)_{n=1}^\infty$, where $x_n\in X_n$, for all $n\in\N$, such that 

$$\|(x_n)_{n=1}^\infty\|\vcentcolon =\Big\|\sum_{n\in\N}\|x_n\|_ne_n\Big\|_E<\infty.$$\hfill

\noindent The space $(\oplus _n X_n)_\mathcal{E}$ endowed with the norm $\|\cdot\|$ defined above is a Banach space.  If the $X_n$'s are all the same, say $X_n=X$, for all $n\in\N$, we write $(\oplus X)_\mathcal{E}$. 

\subsection{Nonlinear embeddings.} \label{SubsectionEmb} Let $X$ be a Banach space and $M$ be a metric space. Then,  a uniformly continuous map $f:X\to M$ is automatically coarse. Moreover, if $f$ is coarse, then there exists $L>0$ such that $\omega_f(t)\leq Lt+L$, for all $t\geq 0$  (see \cite{Ka3}, Lemma 1.4). In particular, $f$ is Lipschitz for large distances. Indeed, if $\|x-y\|\geq L$, we have that $\|f(x)-f(y)\|\leq (L+1)\|x-y\|$. By replacing $f$ by $f(L\ \cdot)/(L+1)$, we can assume that 

$$\|x-y\|\geq 1\ \ \text{  implies }\ \ \|f(x)-f(y)\|\leq \|x-y\|.$$\hfill

The following proposition, proved in \cite{Ro}, Lemma 60, gives us a useful equivalent definition of solvent maps.

\begin{prop}\label{propsolvent}
Let $X$ be a Banach space and $M$ be a metric space. Then a coarse map $f:X\to M$ is solvent if and only if $\sup_{t>0}\overline{\rho}_f(t)=\infty$.  
\end{prop}

Although the statement of the next proposition is different from Proposition 63 of \cite{Ro}, its proof is the same. However, as its proof is very simple and as this result will play an important role in our notes, for the convenience of the reader, we include its proof here.

\begin{prop}\label{Rosendal}
Let $X$ and $Y$ be a Banach space, and let $\mathcal{E}$ be an $1$-unconditional basic sequence.  Assume that there exists a  uniformly continuous  almost uncollapsed  map $\varphi:X\to Y$. Then, there exists a uniformly continuous  solvent map  $\Phi:X\to (\oplus Y)_\mathcal{E}$. 
\end{prop}

\begin{proof}
Let $\varphi :X\to Y$ be a uniformly continuous  almost uncollapsed  map. As $\varphi$ is  almost uncollapsed, pick $t>0$  such that $\overline{\rho}_\varphi(t)>0$. As $\varphi$ is uniformly continuous, pick a sequence of positive reals $(\eps_n)_n$ such that

$$\|x-y\|<\eps_n\ \ \Rightarrow \ \ \|\varphi(x)-\varphi(y)\|<\frac{1}{n2^{n}},$$\hfill

\noindent for all $x,y\in X$. 

For each $n\in\N$, let $\Phi_n(x)=n\cdot \varphi\big(\frac{\eps_n}{n}x\big)$, for all $x\in X$. Then, for  $n_0\in\N$, and  $x,y\in X$, with $\|x-y\|\leq n_0$, we have that 

\begin{align*}
\|\Phi_n(x)-\Phi_n(y)\|&= n\cdot \Big\| \varphi\Big(\frac{\eps_n}{n}x\Big)-\varphi\Big(\frac{\eps_n}{n}y\Big)\Big\|\leq \frac{1}{2^n},	
\end{align*}\hfill

\noindent for all $n\geq n_0$. Define $\Phi:X\to (\oplus Y)_\mathcal{E}$ by letting $\Phi(x)=(\Phi_n(x))_n$, for all $x\in X$.  By the above, $\Phi$ is well-define and it is uniformly continuous. Now notice that, if $\|x-y\|=tn/\eps_n$, then $\|\frac{\eps_n}{n}x-\frac{\eps_n}{n}y\|=t$. Hence,  if $\|x-y\|=tn/\eps_n$, we have that

\begin{align*}
\|\Phi(x)-\Phi(y)\|&\geq \|\Phi_n(x)-\Phi_n(y)\| = n\cdot \Big\|\varphi\Big(\frac{\eps_n}{n}x\Big)-\varphi\Big(\frac{\eps_n}{n}y\Big)\Big\|\geq n\cdot\overline{\rho}_\varphi(t). 	
\end{align*}\hfill

\noindent So, as $\overline{\rho}_\varphi(t)>0$, we have that $\lim_n\overline{\rho}_\Phi(tn/\eps_n)= \infty$. By Proposition \ref{propsolvent}, $\Phi$ is solvent.
\end{proof}

\subsection{Type and cotype.}\label{deftype}

Let $X$ be a Banach space and $p\in (1,2]$ (resp. $q\in [2,\infty)$). We say  that $X$ has \emph{type $p$} (resp. \emph{cotype $q$}) if there exists $T>0$ (resp. $C>0$) such that, for all $x_1,\ldots, x_n\in X$,

$$\E_\eps\Big\|\sum_{j=1}^n\eps_jx_j\Big\|^p\leq T^p\sum_{j=1}^n\|x_j\|^p \ \ \Big(\text{resp. }\E_\eps\Big\|\sum_{j=1}^n\eps_jx_j\Big\|^q\geq \frac{1}{C^q}\sum_{j=1}^n\|x_j\|^q\Big),$$\hfill

\noindent where the expectation above is taken with respect to a uniform choice of signs $\eps=(\eps_j)_j\in\{-1,1\}^n$. The smallest $T$ (resp. $C$) for which this holds is denoted $T_p(X)$ (resp. $C_q(X)$). We say that $X$ has \emph{nontrivial type} (resp. \emph{nontrivial cotype}) if $X$ has type $p$, for some $p\in (1,2]$ (resp. if $X$ has cotype $q$, for some $q\in[2,\infty)$).

\subsection{Cocycles.}\label{subsectionEx} By the Mazur-Ulam Theorem (see \cite{MU}), any surjective  isometry $A:Y\to Y$ 	of a Banach space $Y$ is affine, i.e., there exists a surjective linear isometry $T:Y\to Y$, and some $y_0\in Y$, such that $A(y)=T(y)+y_0$, for all $y\in Y$. Therefore, if $G$ is a group, every isometric action $\alpha:G\curvearrowright Y$  of $G$ on the Banach space $Y$ is an affine isometric action, i.e., there exists an isometric linear action $\pi:G\curvearrowright Y$, and a map $b:G\to Y$ such that

\begin{align*}
\alpha_g(y)=\pi_g(y)+b(g),
\end{align*}\hfill

\noindent for all $g\in G$, and all $y\in Y$. The map $b:G\to Y$ is called the \emph{cocycle of $\alpha$}, and it is given by $b(g)=\alpha_g(0)$, for all $g\in G$. As $\alpha$ is an action by isometries, we have that

\begin{align*}
\|b(g)-b(h)\|&=\|\alpha_g(0)-\alpha_{h}(0)\|=\|\alpha_{h^{-1}g}(0)\|=\|b(h^{-1}g)\|
\end{align*}\hfill

\noindent for all $g,h\in G$. Hence, if $G$ is a metric group,  a continuous cocycle $b:G\to Y$ is automatically uniformly continuous.

\iffalse
 An arbitrary map $b: G\to Y$ is called an \emph{cocycle} if there exists an isometric linear action $\pi:G\curvearrowright Y$ such that Equation \ref{cocycle} defines an affine isometric action on $Y$. In particular, if $b:G\to Y$ is a cocycle, then $\|b(g)-b(h)\|=\|b(h^{-1}g)\|$, for all $g,h\in G$. Hence, if $G$ is a metric group,  a continuous cocycle $b:G\to Y$ is automatically uniformly continuous.

If $(G,d)$ is a topological group with a compatible metric $d$, we call an affine action $\alpha:G\curvearrowright Y$ \emph{coarse} (resp. \emph{Lipschitz}, \emph{expanding}, \emph{solvent}, \emph{collapsed}, \emph{uncollapsed}, or \emph{ almost uncollapsed}) if its cocycle $b:G\to Y$ is coarse (resp. Lipschitz, expanding, solvent, collapsed, uncollapsed, or  almost uncollapsed). 
\fi

\begin{remark}
If $(X,\|\cdot\|)$ is a Banach space,  we look at $(X,+)$ as an additive group with a metric given by the norm $\|\cdot\|$. So, we can work with affine isometric actions $\alpha:X\curvearrowright Y$ of the additive group $(X,+)$ on a Banach space $Y$.
\end{remark}

Let $\alpha:G\curvearrowright Y$ be an action by affine isometries. Its cocycle $b$ is called a  \emph{coboundary} if there exists $\xi\in Y$ such that $b(g)=\xi-\pi_g(\xi)$, for all $g\in G$. Clearly, $b$ is a coboundary if and only if $\alpha$ has a fixed point. Also, if $Y$ is reflexive, then $\text{Im}(b)$ is bounded if and only if $b$ is a coboundary. Indeed, if $b$ is a coboundary, it is clear that $\text{Im}(b)$ is bounded. Say $\text{Im}(b)$ is bounded and let $\mathcal{O}$ be an orbit of the action $\alpha$. Then the closed convex hull $\overline{\text{conv}}(\mathcal{O})$ must be bounded, hence weakly compact (as $Y$ is reflexive). Therefore, by Ryll-Nardzewski fixed-point theorem (see \cite{R-N}), there exists $\xi\in Y$ such that $\alpha_g(\xi)=\xi$, for all $g\in G$. So, $b(g)=\xi-\pi_g(\xi)$, for all $g\in G$. 

The discussion above is well-known, and we isolate it in the proposition below.

\begin{prop}\label{coboundary}
Let $G$ be a group and $Y$ be a Banach space. Let $\alpha:G\curvearrowright Y$ be an action by affine isometries with cocycle $b$. Then $b$ is a coboundary if and only if $\alpha$ has a fixed point. Moreover, if $Y$ is reflexive, then $b$ is a coboundary if and only if $b$ is bounded.
\end{prop}

As we are interested in studying the relations between maps which are expanding, solvent,  uncollapsed, and  almost uncollapsed, it is important to know that those are actually different classes of maps. The next proposition shows that there are maps which are both solvent and collapsed (see \cite{E}, Theorem 2.1, for a similar example). In particular, such maps are not expanding. 

\begin{prop}\label{solventcollapsedmaps}
There exists an affine isometric action $\R\curvearrowright \ell_2(\mathbb{C})$ whose cocycle is  Lipschitz, solvent, and collapsed. 
\end{prop}

\begin{proof}
Define an action $U:\R\curvearrowright \C^\N$ by letting 

$$U_t(x)=\Big(\exp\Big(\frac{2\pi i t}{2^{2^n}}\Big) x_n\Big)_n,$$\hfill

\noindent for all $t\in \R$, and all $x=(x_n)_n\in\C^\N$. Let $w=(1,1,\ldots)\in \C^\N$ and define an action $\alpha:\R\curvearrowright \C^\N$ as $\alpha_t(x)=w+U_t(x-w)$, for all $t\in \R$, and all $x\in\C^\N$. So, 

\begin{align}\label{eqA}
(\alpha_t(x))_m=  \exp\Big(\frac{2\pi i t}{2^{2^m}}\Big)x_m +  1-\exp\Big(\frac{2\pi i t}{2^{2^m}}\Big),
\end{align}\hfill

\noindent for all $t\in \R$,  all $x=(x_n)_n\in\C^\N$, and all $m\in\N$. As $|1-\exp(\theta i)|\leq |\theta|$, for all $\theta\in\R $, it follows that $( 1-\exp(2\pi i t/2^{2^n}))_n\in \ell_2(\C)$, for all $t\in \R$. Hence,   $\alpha_t(x)\in\ell_2(\C)$, for all $ t\in \R$, and all $x\in\ell_2(\C)$. So, $\alpha$ restricts to an action $\alpha:\R\curvearrowright \ell_2(\C)$. By Equation \ref{eqA}, it follows that $\alpha:\R\curvearrowright \ell_2(\C)$ is an affine isometric action. 

Let $b:\R\to \ell_2(\C)$ be the cocycle of $\alpha:\R\curvearrowright \ell_2(\C)$, i.e., $b(t)=\alpha_t(0)$, for all $t\in\R$. Then, an easy induction gives us that $b(t)=w-U_t(w)$, for all $t\in\R$. Let $C=\sum_{n\in\N}\Big(\frac{2\pi}{2^{2^n}}\Big)^2$, then

\begin{align*}
\|b(t)\|^2&=\sum_{n\in\N}\Big|1-\exp\Big(\frac{2\pi i t}{2^{2^n}}\Big)\Big|^2\leq \sum_{n\in\N}\Big(\frac{2\pi t}{2^{2^n}}\Big)^2=C|t|^2,
\end{align*}\hfill

\noindent for all $t\in\R$. So, $b$ is Lipschitz. 

For $t\neq 0$, $0\in \C^\R$ is the only fixed point of $U_t$. Hence,  $w$ is the only fixed point of $\alpha_t$. So, as $w\not\in \ell_2(\C)$, $\alpha:\R\curvearrowright \ell_2(\C) $ has no fixed points. Therefore, $b$ is unbounded (see Proposition \ref{coboundary}). By Proposition \ref{propsolvent}, $b$ is solvent. 

Pick $L>0$ such that $Ls\leq 2^s-1$, for all $s\in\N$. If $k\in\N$ is large enough, say ${2\pi }/{ 2^{2^kL}}<1$, we have that

\begin{align*}
\|b(2^{2^k})\|^2&=\sum_{n>k}\Big| 1-\exp\Big(\frac{2\pi i 2^{2^k}}{2^{2^n}}\Big)\Big|^2\leq \sum_{n>k}\Big(\frac{2\pi 2^{2^k}}{2^{2^n}}\Big)^2\\
&= \sum_{s\in\N}\Big(\frac{2\pi }{ 2^{2^k(2^{s}-1)}}\Big)^2\leq \sum_{s\in\N}\Big(\frac{2\pi }{ 2^{2^kLs}}\Big)^2\\
&\leq \frac{2\pi}{2^{2^kL}-1}.
\end{align*}\hfill

\noindent Hence, $\|b(2^{2^k})\|\to 0$, as $k\to \infty$. So, $b$ is collapsed.
\end{proof}

\begin{problem}
 Is there a map $X\to Y$ which is collapsed,  almost uncollapsed and bounded  (in particular not solvent), for some Banach spaces $X$ and $Y$?
\end{problem}

\section{Preservation of cotype.}\label{sectioncotype}

Mendel and Naor solved in \cite{MN} the long standing problem in Banach space theory  of giving a completely metric definition for the cotype of a Banach space. As a subproduct of this, they have shown that if a Banach space $X$ coarsely (resp. uniformly) embeds into a Banach space $Y$ with nontrivial type, then $q_X\leq q_Y$ (see \cite{MN}, Theorem 1.9 and Theorem 1.11).  In this section we prove Theorem \ref{solventemb}, which  shows that the hypothesis on the embedding $X\hookrightarrow Y$ can be weakened. 

For every $m\in\N$, we denote by $\Z_m$ the set of integers modulo $m$. For every $n,m\in\N$, we denote  the normalized counting measure on $\Z^n_m$ by $\mu=\mu_{n,m}$, and  the normalized counting measure on $\{-1,0,1\}^n$ by $\sigma=\sigma_n$.

\begin{defi}\textbf{(Metric cotype)} Let $(M,d)$ be a metric space and $q,\Gamma>0$. We say that $(M,d)$ has \emph{metric cotype $q$ with constant $\Gamma$} if, for all $n\in\N$, there exists an even integer $m$, such that, for all $f:\Z^n_m\to M$,

\begin{align}\label{metriccotypedef}
\sum_{j=1}^n\int_{\Z^n_m}d\Big(f\big(x&+\frac{m}{2}e_j\big),f(x)\Big)^qd\mu(x)\\
& \leq\Gamma^qm^q\int_{\{-1,0,1\}^n}\int_{\Z^n_m}d\big(f(x+\eps),f(x)\big)^qd\mu(x)d\sigma(\eps).\nonumber
\end{align}\hfill

\noindent  The infimum of the constants $\Gamma$ for which $(M,d)$ has metric cotype $q$ with constant $\Gamma$ is denoted by $\Gamma_q(M)$. Given $n\in\N$ and $\Gamma>0$, we define $m_q(M,n,\Gamma)$ as the smallest even integer $m$ such that Inequality \ref{metriccotypedef} holds, for all $f:\Z^n_m\to M$. If no such $m$ exists we set $m_q(M,n,\Gamma)=\infty$.
\end{defi}

The following is the main theorem of \cite{MN}. Although we will not use this result in these notes, we believe it is worth mentioning.

\begin{thm}\textbf{(Mendel and Naor, 2008)}
Let $X$ be a Banach space and $q\in[2,\infty)$. Then $X$ has metric cotype $q$ if and only if $X$ has cotype $q$. Moreover, 

$$\frac{1}{2\pi}C_q(X)\leq \Gamma_q(X)\leq 90 C_q(X),$$\hfill

\noindent where $C_q(X)$ is the $q$-cotype constant of $X$.
\end{thm}

We start by proving a simple property of solvent maps. 

\begin{prop}\label{ltrivial}
Let $(M,d)$ and $(N,\partial)$ be metric spaces, $\varphi:M\to N$ be a solvent map, and $S>0$. If $[a_n,b_n]_n$ is a sequence of intervals of the real line such that $\lim_n a_n=\infty$, $b_n-a_n<S$ and $a_{n+1}-a_n<S$, for all $n\in\N$, then, we must have

$$\sup_n\inf\{\overline{\rho}_\varphi(t)\mid t\in [a_n,b_n]\}=\infty.$$\hfill
\end{prop}

\begin{proof}
Let $k>0$. Pick $N\in\N$ so that $N\geq \{a_1+S, k, 2S\}$, and let $R\geq 0$ be such that $d(x,y)\in [R,R+N]$ implies $\partial(f(x),f(y))>N$, for all $x,y\in M$. Then there exists $n\in\N$ such that $[a_n,b_n]\subset [R,R+N]$. Indeed, if $a_1< R$ let $n=\max \{j\in\N\mid a_j<R\}+1$, and if $R\leq a_1$ let $n=1$. Hence, 

$$\inf\{\overline{\rho}_\varphi(t)\mid t\in [a_n,b_n]\}\geq \inf\{\overline{\rho}_\varphi(t)\mid t\in [R,R+N]\}\geq N\geq k.$$\hfill

\noindent As $k$ was chosen arbitrarily, we are done.
\end{proof}

The following lemma is a version of Lemma 7.1 of \cite{MN} in the context of the modulus $\overline{\rho}$ instead of $\rho$. It's proof is analogous to the proof of Lemma 7.1 of \cite{MN} but we include it here for completeness. Let $n\in\N$ and  $r\in [1,\infty]$. In what follows,  $\ell_r^n(\C)$ denotes the complex Banach space $(\C^n,\|\cdot\|_r)$, where $\|\cdot\|_r$ denotes the $\ell_r$-norm in $\C^n$.

\begin{lemma}\label{lemmasolvent}
Let $(M,d)$ be a metric space, $n\in\N$, $q,\Gamma>0$, and $r\in [1,\infty]$. Then,  for every map $f:\ell^n_r(\C)\to M$, and every $s>0$, we have

$$n^{1/q}\overline{\rho}_f(2s)\leq \Gamma \cdot m_q(M,n,\Gamma)\cdot\omega_f\left(\frac{2\pi sn^{1/r}}{m_q(M,n,\Gamma)}\right)$$\hfill

\noindent (if $r=\infty$, we use the notation $1/r=0$). 
\end{lemma}

\begin{proof}
In order to simplify notation, let $m=m_q(M,n,\Gamma)$ and assume $r<\infty$ (if $r=\infty$, the same proof holds with the $\ell_r$-norm substituted by the max-norm below). Let $e_1,\ldots,e_n$ be the standard basis of $\ell_r^n(\C)$. Let $h:\Z^n_m\to \ell^n_r(\C)$ be given by 

$$h(x)=s\cdot\sum_{j=1}^ne^{\frac{2\pi i x_j}{m}}e_j,$$\hfill
 
\noindent for all $x=(x_j)_{j}\in \Z^n_m$, and define $g:\Z^n_m\to M$ by letting $g(x)=f(h(x))$, for all $x=(x_j)_{j}\in \Z^n_m$. Then, as 

\begin{align*}
d(g(x+\eps),g(x))\leq \omega_f\Big(s\big(\sum_{j=1}^n|e^{\frac{2\pi i\eps_j}{m}}-1|^r\big)^{1/r}\Big)\leq \omega_f\Big(\frac{2\pi s n^{1/r}}{m}\Big),
\end{align*}\hfill

\noindent for all $\eps=(\eps_j)_j\in \{-1,0,1\}^n$ and all $x=(x_j)_j\in \Z^n_m$, we must have 

\begin{align*}
\int_{\{-1,0,1\}^n}\int_{\Z^n_m}d\big(g(x+\eps),g(x)\big)^qd\mu(x)d\sigma(\eps)\leq  \omega_f\Big(\frac{2\pi s n^{1/r}}{m}\Big)^q.
\end{align*}\hfill

Also, as $\|h(x+\frac{m}{2}e_j)-h(x)\|=2s$, for all $x\in \Z^n_m$, and all $j\in\{1,\ldots,n\}$, we have that $d(g(x+\frac{m}{2}e_j),g(x))\geq \overline{\rho}_f(2s)$, for all $x\in \Z^n_m$, and all $j\in\{1,\ldots,n\}$. Hence, 

$$\sum_{j=1}^n\int_{\Z^n_m}d\Big(g(x+\frac{m}{2}e_j),g(x)\Big)^qd\mu(x)\geq n \overline{\rho}_f(2s)^q.$$\hfill

\noindent Therefore, by the definition of $m_q(M,n,\Gamma)$, we conclude that

$$n\overline{\rho}_f(2s)^q\leq \Gamma^q m^q\omega_f\Big(\frac{2\pi s n^{1/r}}{m}\Big)^q.$$\hfill

\noindent Raising both sides to the $(1/q)$-th power, we are done. 
\end{proof}
  
  We can now prove the main result of this section.

\begin{proof}[Proof of Theorem \ref{solventemb}]
First, let us notice that we only need to prove the case in which $\varphi$ is coarse and solvent. Indeed, let $\varphi:X\to Y$ be a uniformly continuous  almost uncollapsed  map, then $X$ maps into $\ell_2(Y)$ by a uniformly continuous  solvent map (see Proposition \ref{Rosendal}). As $p_{\ell_2(Y)}=p_Y$ and $q_{\ell_2(Y)}=q_Y$, there is no loss of generality if we assume that $\varphi$ is solvent.

If $q_Y=\infty$ we are done, so  assume $q_Y<\infty$. Suppose $q_X>q_Y$. Pick $q\in(q_Y, q_X)$ such that $1/q-1/q_X<1$, and let $\alpha=1/q-1/q_X$ (if $q_X=\infty$, we use the notation $1/q_X=0)$. 

Let $(\eps_n)_n$ be a sequence in $(0,1)$  such that $(1+\eps_n)n^\alpha\leq n^\alpha+1$, for all $n\in\N$. By Maurey-Pisier Theorem (see \cite{MP}), $\ell_{q_X}$ is finitely representable in $X$. Considering $\ell_p(\C)$ as a real Banach space, we have that $\ell_p(\C)$ is finitely representable in $\ell_p$, so $\ell_p(\C)$  is finitely representable in $X$. Therefore, looking at $\ell_p^n(\C)$ as real Banach spaces,   we can pick a (real) isomorphic embedding $f_n:\ell_{q_X}^n(\C)\to X$  such that $\|x\|\leq \|f_n(x)\|\leq (1+\eps_n)\|x\|$,  for all $x\in\ell_{q_X}^n(\C)$. For each $n\in\N$, let $g_n=\varphi\circ f_n$. Hence, 

\begin{align*}
\overline{\rho}_{g_n}(t)&=\inf\{\|\varphi(f_n(x))-\varphi(f_n(y))\|\mid \|x-y\|=t\}\\
&\geq  \inf\{\overline{\rho}_\varphi(\|f_n(x)-f_n(y)\|)\mid \|x-y\|=t\}\\
&\geq \inf\{\overline{\rho}_\varphi (a)\mid a\in[t,(1+\eps_n)t]\},
\end{align*}\hfill

\noindent for all $n\in\N$, and all $t\in[0,\infty)$. Also,  as $\eps_n\in (0,1)$, we have that  $\omega_{g_n}(t)\leq \omega_\varphi(2t)$,  for all $n\in\N$, and all $t\in[0,\infty)$. 

 As $Y$ has nontrivial type and as $q>q_Y$,  Theorem 4.1 of \cite{MN} gives us that, for some $\Gamma>0$,  $m_q(Y,n,\Gamma)=O(n^{1/q})$. Therefore, there exists $A>0$ and $n_0\in\N$ such that $m_q(Y,n,\Gamma)\leq An^{1/q}$, for all $n>n_0$. On the other hand, by Lemma 2.3 of \cite{MN},  $m_q(Y,n,\Gamma)\geq n^{1/q}/\Gamma$, for all $n\in\N$. Hence,  applying Lemma \ref{lemmasolvent} with $s=n^{\alpha}$ and $r=q_X$, we get that, for all $n>n_0$,

$$\inf\{\overline{\rho}_{\varphi}(2a)\mid a\in[2n^{\alpha},2n^{\alpha}+2]\}\leq \Gamma  A \omega_{\varphi}\left(4\pi \Gamma\right).$$\hfill

\noindent As $\alpha<1$, we have that $\sup_n|(n+1)^\alpha-n^\alpha|<\infty$. Therefore, by Proposition \ref{ltrivial}, the supremum over $n$ of the left hand side above is infinite. As $\varphi$ is coarse, this gives us a contradiction. 
\end{proof}

\begin{proof}[Proof of Corollary \ref{funnyOS}]
If $p>1$, this follows straight from Theorem \ref{solventemb},  the fact that $q_{\ell_p}=\max\{2,p\}$ and that $\ell_p$ has nontrivial type. If $p=1$, let $g:\ell_1\to \ell_2$ be a uniform embedding (see \cite{No}, Theorem 5). Then the conclusion of the corollary must hold for the map $g\circ f:\ell_q\to \ell_2$, which implies that it holds for $f$ as well.
\end{proof}

\section{Property $\mathcal{Q}$.}\label{subsectionPropQ}

For each $k\in\N$, let $\mathcal{P}_k(\N)$ denote the set of all subset of $\N$ with exactly $k$ elements. If $\bar{n}\in \mathcal{P}_k(\N)$, we always write $\bar{n}=\{n_1,\ldots,n_k\}$ in increasing order, i.e., $n_1<\ldots<n_k$. We make $\mathcal{P}_k(\N)$ into a graph by saying that two distinct elements $\bar{n}=\{n_1,\ldots,n_k\},\bar{m}=\{m_1,\ldots,m_k\}\in \mathcal{P}_k(\N)$ are connected if they interlace, i.e., if either

$$n_1\leq m_1\leq n_2\leq \ldots\leq n_k\leq m_k\ \ \text{ or }\ \  
 m_1\leq n_1\leq m_2\leq  \ldots m_k\leq n_k.$$\hfill

\noindent  We write $\bar{n}<\bar{m}$ if $n_k<m_1$. We endow $\mathcal{P}_k(\N)$ with the shortest path metric. So, the diameter of $\mathcal{P}_k(\N)$ equals $k$. 

Kalton introduced the following property for metric spaces in \cite{K}, Section 4.  For $\varepsilon,\delta>0$, a metric space $(M,d)$ is said to have \emph{Property $\mathcal{Q}(\varepsilon,\delta)$} if for all $k\in\N$,  and all $f:\mathcal{P}_k(\N)\to M$ with  $\omega_f(1)\leq \delta$, there exists an infinite subset $\M\subset\N$ such that 
 
$$d(f(\bar{n}),f(\bar{m}))\leq \varepsilon,\ \ \text{ for all }\ \ \bar{n}<\bar{m}\in \mathcal{P}_k(\M).$$\hfill

\noindent For each $\varepsilon>0$, we define $\Delta_M(\varepsilon)$ as the supremum of all $\delta>0$ so that $(M,d)$ has Property $\cQ(\varepsilon,\delta)$. For a Banach space $X$, it is clear that there exists $Q_X\geq 0$ such that $\Delta_X(\varepsilon)=Q_X\varepsilon$, for all $\varepsilon>0$. The Banach space $X$ is said to have \emph{Property $\cQ$} if $Q_X>0$. 
 
\begin{theorem}\label{PropertyQ}
Let $X$ and $Y$ be Banach spaces, and assume that $Y$ has Property $\mathcal{Q}$. If either

\begin{enumerate}[(i)]
\item there exists a coarse solvent map $X\to Y$, or
\item  there exists a uniformly continuous  map $\varphi: B_X\to Y$  such that $\overline{\rho}_\varphi(t)>0$, for some $t\in (0,1)$, 
\end{enumerate}

\noindent then, $X$ has Property $\mathcal{Q}$. In particular, if there exists a uniformly continuous almost uncollapsed  map $X\to Y$, then, $X$ has Property $\mathcal{Q}$.
\end{theorem}

\begin{proof}
(i) Assume $\varphi: X\to Y$ is a coarse solvent map. In particular, $\omega_\varphi(1)>0$. Fix $j\in\N$, and pick $R>0$  such that

\begin{align}\label{EqPropQ1}\|x-y\|\in [R,R+j]\ \ \text{ implies }\ \ \|\varphi(x)-\varphi(y)\|>j,
\end{align}\hfill

\noindent for all $x,y\in X$. Assume that $X$ does not have Property $\cQ$. So, $\Delta_X(R)=0$, and  there exists $k\in\N$, and  $f:\mathcal{P}_k(\N)\to X$ with $\omega_{f}(1)\leq 1$, such that, for all infinite $\M\subset \N$, there exists $\bar{n}<\bar{m}\in \mathcal{P}_k(\M)$ such that $ \|f(\bar{n})-f(\bar{m})\|> R$. By standard Ramsey theory (see \cite{T}, Theorem 1.3), we can assume that $\|f(\bar{n})-f(\bar{m})\|>R$, for all $\bar{n}<\bar{m}\in \mathcal{P}_k(\M)$.

Pick a positive $\theta <j$.  As $\omega_{f}(1)\leq 1$, we have that $\|f(\bar{n})-f(\bar{m})\|\in [R,k]$, for all $\bar{n}<\bar{m}\in \mathcal{P}_k(\M)$. Therefore, applying Ramsey theory again, we can get an infinite subset $\M'\subset \M$, and $a\in [R,k]$ such that $\|f(\bar{n})-f(\bar{m})\|\in [a,a+\theta]$, for all $\bar{n}<\bar{m}\in\mathcal{P}_k(\M')$.  By our choice of $\theta$, it follows that 

\begin{align}\label{EqProp2} \Big\|\frac{R}{a}f(\bar{n})-\frac{R}{a}f(\bar{m})\Big\|\in [R,R+j], \ \ \text{ for all }\ \ \bar{n}<\bar{m}\in\mathcal{P}_k(\M').
\end{align}\hfill

Let $Q_Y>0$ be the constant given by the fact that $Y$ has  Property $\cQ$. Let $g=(R/a)f$. As $R/a\leq 1$, we have that $\omega_{\varphi\circ {g}}(1)\leq  \omega_\varphi(1)$. As $\Delta_Y(2\omega_\varphi(1)Q_Y^{-1})= 2\omega_\varphi(1) $, we get that there exists $\M''\subset \M'$ such that 

\begin{align}\label{EqPropQ3}
\|\varphi(g(\bar{n}))-\varphi(g(\bar{m}))\|\leq 2 \omega_\varphi(1) Q_Y^{-1},
\end{align}\hfill

\noindent  for all $\bar{n}<\bar{m}\in \mathcal{P}_k(\M'')$. As $j$ was chosen arbitrarily, (\ref{EqPropQ1}), (\ref{EqProp2}) and (\ref{EqPropQ3}) above gives us that $j<2 \omega_\varphi(1) Q_Y^{-1}$, for all $j\in\N$. As $\varphi$ is coarse, this gives us a contradiction.

(ii)  Assume $\varphi: B_X\to Y$ is a uniformly continuous map, and let $t\in (0,1)$ be such that $\overline{\rho}_\varphi(t)>0$. As $\varphi$ is uniformly continuous, we can pick $\rho\in (t,1)$, $s,r\in (0,\rho)$ with $s<t<r$, and  $\gamma>0$, such that

\begin{align}\label{EqPropQ4}
\|x-y\|\in [s,r]\ \ \text{ implies }\ \ \|\varphi(x)-\varphi(y)\|>\gamma,
\end{align}
\hfill

\noindent for all $x,y\in \rho B_X$. Assume that $X$ does not have Property $\cQ$. So, $\Delta_{X}(s)=0$. Fix $j\in\N$.  Then, there exists $k\in\N$, and $f:\mathcal{P}_k(\N)\to X$ with $\omega_{f}(1)\leq j^{-1}$, such that, for all infinite $\M\subset \N$, there exists $\bar{n}<\bar{m}\in \mathcal{P}_k(\M)$ such that $ \|f(\bar{n})-f(\bar{m})\|> s$. Without loss of generality, we can assume that  $ \|f(\bar{n})-f(\bar{m})\|> s$, for all $\bar{n}<\bar{m}\in\mathcal{P}_k(\M)$.  

Pick a positive $\theta<(r-s)$. As $ \|f(\bar{n})-f(\bar{m})\|\in [s,k]$, we can use  Ramsey theory once again to pick an infinite $\M'\subset\M$, and $a\in [s,k]$ such that $ \|f(\bar{n})-f(\bar{m})\|\in [a,a+\theta]$, for all $\bar{n}<\bar{m}\in \mathcal{P}_k(\M')$. By our choice of $\theta$, it follows that 

\begin{align}\label{EqPropQ5} \Big\|\frac{s}{a}f(\bar{n})-\frac{s}{a}f(\bar{m})\Big\|\in [s,r], \ \ \text{ for all }\ \ \bar{n}<\bar{m}\in\mathcal{P}_k(\M').
\end{align}\hfill

Let $\bar{m}_0$ be the first $k$ elements of $\M'$, and $\M''=\M'\setminus \bar{m}_0$. For each $\bar{n}\in\mathcal{P}_k(\M'')$, let  $h(\bar{n})=(s/a)(f(\bar{n})-f(\bar{m}_0))$. Then,  $h(\bar{n})\in \rho B_X$, and $\|h(\bar{n})-h(\bar{m})\|\in [s,r]$, for all $\bar{n}<\bar{m}\in\mathcal{P}_k(\M'')$. As $s/a\leq 1$, we have $\omega_h(1)\leq \omega_f(1)$. Hence, $\omega_{\varphi\circ {h}}(1)\leq \omega_\varphi(j^{-1})$.

Let $Q_Y>0$ be the constant given by the fact that $Y$ has  Property $\cQ$. As $\Delta_Y(2 \omega_\varphi(j^{-1}) Q_Y^{-1})= 2 \omega_\varphi(j^{-1}) $,  there exists $\M'''\subset \M''$ such that 

\begin{align}\label{EqPropQ6}
\|\varphi(h(\bar{n}))-\varphi(h(\bar{m}))\|\leq 2 \omega_\varphi(j^{-1}) Q_Y^{-1},
\end{align}\hfill

\noindent  for all $\bar{n}<\bar{m}\in \mathcal{P}_k(\M''')$. As $j$ was chosen arbitrarily, (\ref{EqPropQ4}), (\ref{EqPropQ5}) and (\ref{EqPropQ6}) gives us that $\gamma<2 \omega_\varphi(j^{-1}) Q_Y^{-1}$, for all $j\in\N$. As $\varphi$ is uniformly continuous, this gives us a contradiction.
\end{proof}

We can now prove the following generalization of Theorem 5.1 of \cite{K}.

\begin{theorem}\label{IntoSuperreflexive2}
Let $X$ and $Y$ be Banach spaces, and assume that $Y$ is reflexive (resp. super-reflexive). If either

\begin{enumerate}[(i)]
\item there exists a coarse solvent map $X\to Y$, or
\item there exists a uniformly continuous map $\varphi:B_X\to Y$ such that $\overline{\rho}_\varphi(t)>0$, for some $t\in (0,1)$,
\end{enumerate}

\noindent then, $X$ is either reflexive (resp. super-reflexive) or $X$ has a spreading model equivalent to the $\ell_1$-basis (resp. trivial type). 
\end{theorem}

\begin{proof}
By Corollary 4.3 of \cite{K}, any reflexive Banach space has Property $\cQ$. By Theorem 4.5 of \cite{K}, a Banach space with Property $\cQ$ must be either reflexive or have a spreading model equivalent to the $\ell_1$-basis (in particular, have nontrivial type). Therefore, if $Y$ is reflexive, the result now follows from Theorem \ref{PropertyQ}.  

 For an index set $I$ and an ultrafilter $\mathcal{U}$ on $I$, denote by $X^I/\mathcal{U}$ the ultrapower of $X$ with respect to  $\mathcal{U}$. Say $Y$ is super-reflexive. In particular, by Corollary 4.3 of \cite{K}, every ultrapower of $Y$ has Property $\mathcal{Q}$. If $X$ maps into $Y$ by a  coarse and solvent map, then $X^I/\mathcal{U}$ maps into $Y^I/\mathcal{U}$ by a coarse and solvent map.  Therefore, it follows from Theorem \ref{PropertyQ} that every ultrapower of $X$ has Property $Q$. Suppose $X$ has nontrivial type. Then, all ultrapowers of $X$ have nontrivial type. Therefore, by Theorem 4.5 of \cite{K}, we conclude that all ultrapowers of $X$ are reflexive. Hence, item (i)  follows. 

Similarly, if there exists $\varphi:B_X\to Y$ as in item (ii), then the unit balls of ultrapowers of $X$ are mapped into ultrapowers of $Y$ by maps with the same properties as $\varphi$, and item (ii) follows. 
\end{proof}

\begin{proof}[Proof of Theorem \ref{IntoSuperreflexive}]
Item (ii) of Theorem \ref{IntoSuperreflexive} follows directly from item (ii) of Theorem \ref{IntoSuperreflexive2}.
\end{proof}

\begin{remark}\label{Gideon}
The statement in Theorem \ref{IntoSuperreflexive} cannot be improved so that if $X$ embeds into a super-reflexive space, then $X$ is either super-reflexive or it has an $\ell_1$-spreading model. Indeed, it was proven in Proposition 3.1 of \cite{NaorSchechtman} that $\ell_2(\ell_1)$ strongly embeds into $L_p$, for all $p\geq 4$.  As $(\oplus_n\ell_1^n)_{\ell_2}\subset \ell_2(\ell_1)$, it follows that $(\oplus_n\ell_1^n)_{\ell_2}$ strongly embeds into $L_4$. However  $(\oplus_n\ell_1^n)_{\ell_2}$ is neither super-reflexive nor contains an $\ell_1$-spreading model.
\end{remark}

\section{Embeddings into Hilbert spaces.}\label{SectionHilbert}

In \cite{Ra}, Randrianarivony showed that a Banach space $X$ coarsely embeds into a Hilbert space if and only if it uniformly embeds into a Hilbert space. This result together with Theorem 5 of \cite{No}, gives a positive answer to Problem \ref{mainproblem} for $Y=\ell_p$, for $p\in[1,2]$. In this section, we show that Problem \ref{mainproblemPartII} also has a positive answer if $Y$ is $\ell_p$, for any $p\in [1,2]$. 

First, let us prove a simple lemma. For $\delta>0$, a subset $S$ of a metric space $(M,d)$ is called \emph{$\delta$-dense} if $d(x,S)< \delta$, for all $x\in M$.

\begin{lemma}\label{SolventMapNet}
Let $(M,d)$ and $(N,\partial)$ be Banach spaces and $S\subset M$ be a $\delta$-dense set, for some $\delta>0$. Let $f:M\to N$ be a coarse map such that $f_{|S}$ is solvent. Then $f$ is solvent.
\end{lemma}

\begin{proof}
Let $n\in\N$. As $f_{|S}$ is solvent and $\omega_f(\delta)<\infty$, we can  pick $R>0$ such that 

$$d(x,y)\in [R-2\delta,R+n+2\delta]\ \ \text{ implies }\ \ \partial(f(x),f(y))>n+2\omega_f(\delta),$$\hfill 

\noindent for all $x,y\in S$. Pick $x,y\in X$, with $d(x,y)\in [R,R+n]$. As $S$ is $\delta$-dense, we can pick $x',y'\in S$ such that $d(x,x')\leq \delta$ and $d(y,y')\leq\delta$.  Hence,   $d(x',y')\in [R-2\delta,R+n+2\delta]$, which gives us that $\partial(f(x'),f(y'))>n+2\omega_f(\delta)$. Therefore,  we conclude that $\partial(f(x),f(y))> n$. 
 \end{proof}

The next lemma is an adaptation of Proposition 2 of \cite{Ra}, and its proof is analogous to the proof of Theorem 1 of \cite{JohnsonRandrianarivony}. Before stating the lemma, we need the following definition: a map $K:X\times X\to \R$ is called a \emph{negative definite kernel} (resp. \emph{positive definite kernel}) if 

\begin{enumerate}[(i)]
\item  $K(x,y)=K(y,x)$, for all $x,y\in X$, and 
\item $\sum_{i,j} K(x_i,x_j)c_ic_j\leq 0$ (resp. $\sum_{i,j} K(x_i,x_j)c_ic_j\geq 0$), for  all $n\in\N$, all $x_1,\ldots,x_n\in X$, and all $c_1,\ldots, c_n\in \R$, with $\sum_i c_i=0$.
\end{enumerate}

\noindent  A function $f: X\to R$ is called \emph{negative definite} (resp. \emph{positive definite}) if $K(x,y)=f(x-y)$ is a negative definite kernel (resp. positive definite kernel). 

\begin{lemma}\label{LemmaJR}
Let $X$ be a Banach space and assume that $X$ maps into a Hilbert space by a map which is coarse and solvent. Then there exist $\alpha>0$, a map $\overline{\rho}:[0,\infty)\to [0,\infty)$, with $\limsup_{t\to\infty}\overline{\rho} (t)=\infty$, and  a continuous negative definite function $g: X\to \R$ such that 

\begin{enumerate}[(i)]
\item $g(0)=0$, and
\item $\overline{\rho}(\|x\|)\leq g(x)\leq \|x\|^{2\alpha}$, for all $x\in X$.
\end{enumerate}
\end{lemma}

\begin{proof}[Sketch of the Proof of Lemma \ref{LemmaJR}.]
Let $H$ be a Hilbert space and consider a coarse solvent map $f: X\to H$. Without loss of generality, we may assume that $\|f(x)-f(y)\|\leq \|x-y\|$, for all $x,y\in X$, with $\|x-y\|\geq 1$ (see Subsection \ref{SubsectionEmb}). \\

\textbf{Claim 1:} Let $\alpha\in (0,1/2)$. Then $X$  maps into a Hilbert space by a map which is $\alpha$-H\"{o}lder and solvent.\\

As $H$ is Hilbert,  the assignment $(x, y)\mapsto\|f(x) - f(y) \|^2$ is a negative definite kernel on X (this is a simple computation and it is contained in the proof of  Proposition 3.1 of \cite{Nowak2005}). Hence, for all $\alpha\in (0,1)$, the kernel  $N(x, y) = \|f(x) - f(y)\|^{2\alpha}$
 is also  negative definite (see \cite{Nowak2005}, Lemma 4.2).  So, there exists a Hilbert space $H_\alpha$ and a map $f_\alpha: X\to H_\alpha$ such that $N(x,y)=\|f_\alpha(x)-f_\alpha(y)\|^2$, for all $x,y\in X$ (see \cite{Nowak2005}, Theorem 2.3(2)). This gives us that
 
 $$\big(\overline{\rho}_f(\|x-y\|)\big)^\alpha\leq \|f_\alpha(x)-f_\alpha(y)\|\leq \|x-y\|^\alpha,$$\hfill

\noindent for all $x,y\in X$, with $\|x-y\|\geq 1$.   In particular, $f_{\alpha}$ is solvent. Hence, if $N\subset X$ is a $1$-net (i.e., a maximal $1$-separated set), the restriction $f_{\alpha|N}:N\to H_\alpha$ is $\alpha$-H\"{o}lder and solvent. Using that $\alpha\in (0,1/2)$, Theorem 19.1 of \cite{WW} gives us that there exists an $\alpha$-H\"{o}lder map $F_\alpha:X\to H_\alpha$ extending $f_{\alpha|N}$. By Lemma \ref{SolventMapNet}, $F_\alpha$ is also solvent. This finishes the proof of Claim 1. 

By Claim 1 above, we can assume that $f:X\to H$ is an $\alpha$-H\"{o}lder solvent map, with $\alpha\in (0,1/2)$. Set $N(x,y)=\|f(x)-f(y)\|^2$, for all $x,y\in X$. So, $N$ satisfies 

\begin{align}\label{eqN}
\big(\overline{\rho}_f(\|x-y\|)\big)^2\leq N(x,y)\leq \|x-y\|^{2\alpha},
\end{align}\hfill

\noindent for all $x,y\in X$. Let $\mu$ be an invariant mean on the bounded functions $X\to \R$ (see \cite{BL}, Appendix C, for the definition of an invariant mean, and \cite{BL}, Theorem C.1, for the existence of such invariant mean), and define 

$$g(x)=\int_XN(y+x,y)d\mu(y),\ \ \text{ for all }\ \ x\in X.$$\hfill

\noindent Let $\overline{\rho}(t)=(\overline{\rho}_f(t))^2$, for all $t\geq 0$. As $\int_X 1d\mu=1$, Inequality \ref{eqN} gives us that items (i) and (ii) are satisfied. As $f$ is solvent, we also have that $
\limsup_{t\to \infty}\overline{\rho}(t)=\infty$. The proof that $g$  is a negative definite kernel is contained in Step 2 of \cite{JohnsonRandrianarivony} and the proof that $g$ is continuous is contained in Step 3 of \cite{JohnsonRandrianarivony}. As both proofs are simple computations, we omit them here. 
\end{proof}

We can now prove the main theorem of this section. For that, given a probability space $(\Omega,\mathcal{A},\mu)$, we denote  by $L_0(\mu)$ the space of all measurable functions $\Omega\to \C$ with metric determined by convergence in probability. 

\begin{thm}\label{ThmHilbertCru}
Let $X$ be a Banach space. Then the following are equivalent.

\begin{enumerate}[(i)]
\item $X$ coarsely embeds into a Hilbert space.
\item $X$ uniformly embeds into a Hilbert space.
\item $X$ strongly embeds into a Hilbert space.
\item $X$ maps into a Hilbert space by a map which is coarse and solvent.
\item $X$ maps into a Hilbert space by a map which is uniformly continuous and almost uncollapsed.
\item There is a probability space $(\Omega, \mathcal{A}, \mu)$ such that $X$ is linearly isomorphic to a
subspace of $L_0(\mu)$. 
\end{enumerate}
\end{thm}

\begin{proof}
We only need to show that (iv) implies (vi). Indeed, the equivalence between (i), (ii), and (vi) were established in \cite{Ra}, Theorem 1 (see the paragraph preceeding Theorem 1 of \cite{Ra} as well). By \cite{Ro}, Theorem 2, if $X$ uniformly embeds into a Hilbert space $H$ then $X$ strongly embeds into $\ell_2(H)$. Hence, (ii) and (iii) are also equivalent. Using  Proposition \ref{Rosendal}  with $\mathcal{E}$ being the standard basis of $\ell_2$, we get that  (v) implies (iv). Hence, once we show that (iv) implies (vi), all the equivalences will be established. 

Let $H$ be a Hilbert space and $f:X\to H$ be a coarse solvent map. Let $\alpha>0$, $\overline{\rho}$ and $g:X\to \R$ be given by  Lemma \ref{LemmaJR}. Define $F(x)=e^{-g(x)}$, for all $x\in X$. So, $F$ is a positive definite function (see \cite{Nowak2005}, Theorem 2.2). As $F$ is also continuous, by  Lemma 4.2 of \cite{AMM} applied to $F$,  there exist a probability space $(\Omega,\mathcal{A},\mu) $ and a continuous linear operator $U:X\to L_0(\mu)$ such that 

$$F(tx)=\int_\Omega e^{itU(x)(w)}d\mu(w), \ \ \text{for all} \ \ t\in \R, \ \ \text{and all}\ \ x\in X.$$\hfill

As $U$ is continuous, we only need to show that $U$ is injective and its inverse is continuous. Suppose false. Then there exists a sequence  $(x_n)_n$ in the unit sphere of $X$ such that $\lim_nU(x_n)=0$. By the definition of convergence in $L_0(\mu)$, this gives us that $\lim_nF(tx_n)=1$, for all $t\in \R$. As $\limsup_{t\to \infty}\overline{\rho}(t)=\infty$, we can pick $t_0>0$ such that $e^{-\overline{\rho}(t_0)}<1/2$. Hence, we have that

$$
F(t_0x_n)=e^{-g(t_0x_n)}\leq e^{ -\overline{\rho}(\|t_0x_n\|)}= e^{ -\overline{\rho}(t_0)}<\frac{1}{2}, \ \ \text{ for all } \ \  n\in\N.$$\hfill

\noindent As $\lim_nF(t_0x_n)=1$, this gives us a contradiction.
\end{proof}

\begin{proof}[Proof of Theorem \ref{ThmHilbert}.]
This is a trivial consequence of Theorem \ref{ThmHilbertCru} and the equivalence between coarse and uniform embeddability into $\ell_p$, for $p\in[1,2]$ (see \cite{No}, Theorem 5).
\end{proof}

\section{Embeddings into $\ell_\infty$.}\label{Sectionlinfty}

Kalton proved in \cite{Ka4}, Theorem 5.3, that uniform embeddability into $\ell_\infty$, coarse embeddability into $\ell_\infty$ and Lipschitz embeddability into $\ell_\infty$ are all equivalent. In this section, we show that Problem \ref{mainproblemPartII} also has a positive answer if $Y=\ell_\infty$. 

The following lemma is Lemma 5.2 of \cite{Ka4}. Although in \cite{Ka4} the hypothesis on the map are stronger, this is not used in their proof.

\begin{lemma}\label{SonventLipschitz}
Let $X$ be a Banach space and assume that there exists a Lipschitz  map $X\to \ell_\infty$ that is also almost uncollapsed. Then $X$ Lipschitz embeds into $\ell_\infty$.  
\end{lemma}

\begin{proof}
Let $f:X\to \ell_\infty$ be a Lipschitz almost uncollapsed map. Pick $t>0$ such that $\overline{\rho}_f(t)>0$. Define a map $F:X\to \ell_\infty(\Q_+\times \N)$ by setting $F(x)(q,n)=q^{-1}f(qx)_n$, for all $x\in X$, and all $(q,n)\in \Q_+\times \N$. Then

$$\|F(x)-F(y)\|=\sup_{(q,n)\in \Q_+\times \N}q^{-1}\big|f(qx)_n-f(qy)_n\big|\leq \text{Lip} (f)\cdot\|x-y\|.$$\hfill

\noindent So, $F$ is also Lipschitz. Now notice that, as $f$ is continuous, we have that  

$$\|F(x)-F(y)\|=\sup_{q>0}q^{-1}\|f(qx)-f(qy)\|.$$\hfill

\noindent Hence, if $x\neq y$, by letting  $q=t\|x-y\|^{-1}$, we obtain that

$$\|F(x)-F(y)\|\geq  \frac{\|x-y\|}{t}\cdot\Big\|f\Big(\frac{tx}{\|x-y\|}\Big)-f\Big(\frac{ty}{\|x-y\|}\Big)\Big\|\geq \frac{\overline{\rho}_f(t)}{t}\cdot\|x-y\|.$$\hfill

\noindent So, $F$ is a Lipschitz embedding. 
\end{proof}

\begin{proof}[Proof of Theorem \ref{Thmlinfty}.]
By Theorem 5.3 of \cite{Ka4}, items (i), (ii) and (iii) of Problem \ref{mainproblem} are all equivalent. Using Proposition \ref{Rosendal} with $\mathcal{E}$ being the standard basis of $c_0$, we have that item (v) of Problem \ref{mainproblemPartII} implies item  (iv) of Problem \ref{mainproblemPartII}. Hence, we only need to show that  item (iv) of Problem \ref{mainproblemPartII} implies that $X$ Lipschitz embeds into $\ell_\infty$. For that, let $f: X\to \ell_\infty$ be a coarse solvent map. Without loss of generality, we may assume that $\|f(x)-f(y)\|\leq \|x-y\|$, for all $x,y\in X$, with $\|x-y\|\geq 1$. Let $N\subset X$ be a $1$-net. Then $f_{|N}$ is $1$-Lipschitz and solvent. Recall that  $\ell_\infty$ is a \emph{$1$-absolute Lipschitz retract}, i.e., every Lipschitz map $g:A\to \ell_\infty$, where $M$ is a metric space and $A\subset M$, has a $\text{Lip}(g)$-Lipschitz extension (see \cite{Ka3}, Subsection 3.3). Let $F$ be a Lipschitz extension of $f_{|N}$.  By Lemma \ref{SolventMapNet}, $F$ is solvent. Hence, by Lemma \ref{SonventLipschitz}, it follows that $X$ Lipschitz embeds into $\ell_\infty$. 
\end{proof}

\section{Open questions.}

Besides Problem \ref{mainproblem} and Problem \ref{mainproblemPartII}, there are many other interesting questions regarding those weaker kinds of embeddings. We mention a couple of them in this section.

Raynaud proved in \cite{Raynaud1983} (see the corollary in page 34 of \cite{Raynaud1983}) that if a Banach space $X$ uniformly embeds into a superstable space (see \cite{Raynaud1983} for definitions), then $X$ must contain an $\ell_p$, for some $p\in[1,\infty)$. Hence, in the context of those weaker embeddings, it is natural to ask the following.

\begin{problem}
Say an infinite dimensional Banach space $X$ maps into a superstable space by a map which is both uniformly continuous and almost uncollapsed. Does it follow that $X$ must contain $\ell_p$, for some $p\in [1,\infty)$.
\end{problem}

Similarly, if was proved in \cite{BragaSwift} that  if a Banach space $X$ coarsely embeds into a superstable space, then $X$ must contain an $\ell_p$-spreading model, for some $p\in[1,\infty)$. We ask the following.

\begin{problem}
Say an infinite dimensional Banach space $X$ maps into a superstable space by a map which is both coarse and solvent. Does it follow that $X$ must contain an $\ell_p$-spreading model, for some $p\in [1,\infty)$.
\end{problem}

The properties of a map being  solvent (resp. almost uncollapsed) are not necessarily stable under Lipschitz isomorphisms. Hence, the following question seems to be really important for the theory of solvent (resp. almost uncollapsed) maps between Banach spaces. 

\begin{problem}
Assume that there is no coarse solvent (resp. uniformly continuous almost uncollapsed) map $X\to Y$. Is this also true for any renorming of $X$?
\end{problem}

At last, we would like to notice that we have no results for maps $X \to Y$ which are coarse and almost uncollapsed. Hence, we ask the following.

\begin{problem}
What can we say if $X$ maps into $Y$ by a map which is coarse and almost uncollapsed map? Is this enough to obtain any restriction in the geometries of $X$ and $Y$?
\end{problem}

\noindent \textbf{Acknowledgments:} The author would like to thank his adviser C. Rosendal for all the help and attention he gave to this paper. The author would also like to thank  G. Lancien for suggesting to look at Kalton's Property $\cQ$.


\begin{thebibliography}{99}

\bibitem[AMMi]{AMM}
I. Aharoni, B. Maurey, and B. S. Mityagin, \emph{Uniform embeddings of metric spaces and of
Banach spaces into Hilbert spaces}, Israel J. Math. 52 (1985), no. 3, 251-265.

\bibitem[AlB]{AB}
F. Albiac, and F. Baudier, \emph{Embeddability of snowflaked metrics with applications to the nonlinear geometry of the spaces $L_p$ and $\ell_p$ for $0<p<\infty$}, J. Geom. Anal. 25 (2015), no. 1, 1-24.

\bibitem[BeL]{BL}
Y. Benyamini, and J. Lindenstrauss, \emph{Geometric Nonlinear Functional Analysis}, Coll. Pub.
 48, Amer. Math. Soc., Providence, RI, 2000. 

\bibitem[Br]{B}
B. M. Braga, \emph{Coarse and uniform embeddings}, J. Funct. Anal. 272 (2017), no. 5, 1852-1875. 

\bibitem[BrSw]{BragaSwift}
B. M. Braga, and A. Swift, \emph{Coarse embeddings into superstable spaces}, in preparation. 

\bibitem[E]{E}
M. Edelstein, \emph{On non-expansive mappings of Banach spaces}, Proc. Camb. Philos. Soc. 60 (1964), 439-447.



\bibitem[K]{K}
N. Kalton, \emph{Coarse and uniform embeddings into reflexive spaces}, Quart. J. Math.  58 (2007), 393-414.

\bibitem[K2]{Ka3}
N. Kalton, \emph{The nonlinear geometry of Banach spaces}, Rev. Mat. Complut. 21 (2008), 7-60. 


\bibitem[K3]{Ka4}
N. Kalton, \emph{Lipschitz and uniform embeddings into $\ell_\infty$}, Fund. Math. 212 (2011), 53-69.



\bibitem[K4]{Ka}
N. Kalton, \emph{The uniform structure of Banach spaces}, Math. Ann. 354 (2012), 1247-1288.



\bibitem[JR]{JohnsonRandrianarivony}
W.  B. Johnson, and N. L. Randrianarivony, \emph{$\ell_p$ ($p>2$) does not coarsely embed into a Hilbert space,} Proc. Amer. Math. Soc. 134 (2006), no. 4, 1045-1050.

\bibitem[MP]{MP}
B. Maurey, and G. Pisier, \emph{S\'{e}ries de variables al\'{e}atoires vectori\'{e}lles independantes et priopri\'{e}t\'{e}s g\'{e}om\'{e}triques des espaces de Banach}, Studia Math. 58 (1976), 45-90.

\bibitem[MaU]{MU}
S. Mazur, and S. Ulam, \emph{Sur les transformations isom\'{e}triques d'espaces vectoriels norm\`{e}s}, C.
R. Acad. Sci. Paris 194 (1932), 946-948.

\bibitem[MeN]{MN2004}
M. Mendel, and A. Naor, \emph{ Euclidean quotients of finite metric spaces}, Adv. Math. 189 (2004), 451-494.

\bibitem[MeN2]{MN}
M. Mendel, and A. Naor, \emph{Metric cotype}, Ann. Math.  168 (2008), no. 1, 247-298. 

\bibitem[NS]{NaorSchechtman}
A. Naor, and G. Schechtman, \emph{Pythagorean powers of hypercubes}, Ann. Inst. Fourier  66 (2016), no. 3, 1093-1116.

\bibitem[No1]{Nowak2005}
P. Nowak, \emph{Coarse embeddings of metric spaces into Banach spaces}, Proc. Amer. Math. Soc. 133 (2005), no. 9, 2589-2596.

\bibitem[No2]{No}
P. Nowak, \emph{On coarse embeddability into $\ell_p$-spaces and a conjecture of Dranishnikov}, Fundamenta Mathematicae 189 (2006), no. 2, 111-116.

\bibitem[OS]{OS}
E. Odell, and Th. Schlumprecht, \emph{The Distortion Problem}, Acta Math. 173 (1994), no. 2, 259-281.


\bibitem[R]{Ra}
N. L. Randrianarivony, \emph{Characterization of quasi-Banach spaces which coarsely embed into a Hilbert space}, Proc.
Amer. Math. Soc. 134  (2006), no. 5, 1315-1317.

\bibitem[R]{Raynaud1983}
Y. Raynaud, \emph{Espaces de Banach superstables, distances stables et homéomorphismes uniformes},  Israel J. Math. 44 (1983), no. 1, 33-52.

\bibitem[Ro]{Ro}
C. Rosendal, \emph{Equivariant geometry of Banach spaces and topological groups}, arXiv:1610.01070.

\bibitem[R-N]{R-N}
C. Ryll-Nardzewski, \emph{Generalized random ergodic theorems and weakly almost periodic functions}, Bull. Acad. polon. Sci. S\'{e}r. Sci. Math. Astronom. Phys. 10 (1962), 271-275.




\bibitem[T]{T}
S. Todorcevic, \emph{Introduction to Ramsey spaces}, Annals of Mathematics Studies, 174, Princeton University Press, Princeton, NJ, 2010.

\bibitem[WWi]{WW}
 J. H. Wells, and L. R. Williams, \emph{Embeddings and extensions in Analysis}, Ergebnisse der
 Mathematik und ihrer Grenzgebiete 84, Springer-Verlag, New York-Heidelberg, 1975.
 

\end{thebibliography}
\end{document}